\documentclass[12pt, english, a4paper]{amsart}

\usepackage{graphicx}

\usepackage{amsmath,amssymb}
\usepackage{bm}
\usepackage{euscript,color}
\usepackage{babel}
\usepackage{amsthm}
\usepackage{amsfonts}
\usepackage{latexsym}
\usepackage{fancyhdr}
\usepackage{enumerate}

\usepackage[left=2.5cm, top=3cm, bottom=3cm, right=2.5cm]{geometry}

\def\mymathfont{\mathbf}

\newcommand{\Q}{\mymathfont{Q}}
\newcommand{\R}{\mymathfont{R}}

\newcommand{\Z}{\mymathfont{Z}}

\newcommand{\CP}{\mathop{\mymathfont{C}\mathrm{P}}\nolimits}
\newcommand{\RP}{\mathop{\mymathfont{R}\mathrm{P}}\nolimits}

\makeatletter
\renewcommand{\section}{\@startsection%
{section}
{1}
{0mm}
{1.5\bigskipamount}
{0.5\bigskipamount}
{\centering\normalsize\sc}}

\renewcommand{\paragraph}{\@startsection%
{paragraph}
{4}
{0mm}
{\bigskipamount}
{-1.25ex}
{\normalsize\sl}}

\def\provedboxcontents#1{$\square$}

\makeatother

\newtheoremstyle{thm}{}{}{\slshape}{}{\scshape}{.}{0.5em}{}

\newtheoremstyle{def}{}{}{}{}{\scshape}{.}{0.5em}{}

\newtheoremstyle{rmk}{}{}{}{}{\scshape}{.}{0.5em}{}

\newtheoremstyle{claim}{}{}{}{}{\slshape}{.}{0.5em}{}

\theoremstyle{thm}
\newtheorem{newstatement}{newstatement}

\newtheorem{theorem}[newstatement]{Theorem}
\newtheorem{corollary}[newstatement]{Corollary}

\newtheorem*{conjecture*}{Conjecture}

\theoremstyle{def}

\newcommand{\K}{K\"{a}hler}

\theoremstyle{rmk}

\theoremstyle{claim}

\expandafter\let\expandafter\oldproof\csname\string\proof\endcsname
\let\oldendproof\endproof
\renewenvironment{proof}[1][\proofname]{%
  \oldproof[\slshape #1]%
}{\oldendproof}

\let\geq\geqslant
\let\leq\leqslant

\let\phi\varphi
\let\epsilon\varepsilon

\renewcommand{\emph}[1]{{\slshape #1}}
\renewcommand{\em}{\sl}

\title[Submanifolds of the projective space]{On codimension-1 submanifolds of the real and complex  projective space}

\author{Beniamino Cappelletti--Montano}
\address{Beniamino Cappelletti--Montano, Dipartimento di Matematica e Informatica \\
         U\-ni\-ver\-si\-t\`a di Cagliari, Italy.}
         \email{b.cappellettimontano@unica.it}

\author{Andrea Loi}
\address{Andrea Loi, Dipartimento di Matematica e Informatica \\
         Universit\`a di Cagliari, Italy.}
         \email{loi@unica.it}

\author{Daniele Zuddas}
\address{Daniele Zuddas, Korea Institute for Advanced Study, School of Mathematics, 85 Hoegiro, Dongdaemun-gu, Seoul 02455, Republic of Korea.}
\email{zuddas@kias.re.kr}

\date{}

\thanks{The first two authors were   supported by Prin 2015 -- Real and Complex Manifolds; Geometry, Topology and Harmonic Analysis -- Italy. All the authors are  members of INdAM-GNSAGA - Gruppo Nazionale per le Strutture Algebriche, Geometriche e le loro Applicazioni.}

\begin{document}
\begin{abstract}
Inspired by the analogous result in the algebraic setting (Theorem \ref{mainteor1}) we
show (Theorem \ref{mainteor2}) that the product  $M \times \RP^n$ of a closed and orientable topological manifold $M$ with the  $n$-dimensional real  projective space cannot be topologically locally flat embedded into  $\RP^{m + n + 1}$ for all even $n > m$.
\end{abstract}
\maketitle
\section{Introduction}
Since the early papers of Whitney (\cite{W36}), Hirsch (\cite{H59}) and Smale (\cite{S59}), the theory of embeddings focused on the question whether a smooth manifold can be immersed / embedded into a Euclidean space. In particular, there are a number of results concerning the best immersion or embedding of a projective space of a given dimension into a Euclidean space, and, on the other hand,  topological obstructions to such immersability in some dimensions have been found (see, for instance, \cite{AF13}, \cite{BB65}, \cite{DZ00}, \cite{R70}, \cite{S04}).

In last years the study of immersions has also fascinating interplays with other areas of mathematics. For instance, determining the smallest immersion dimension of a projective space into a Euclidean space is related to the study of robot motion planning \cite{FTY03}.

Very few is known, so far, about embeddings \emph{into} a projective space, a part of some results concerning the analytical and smooth embeddings of $\CP^{m}$  into  $\CP^{n}$ \cite{F65} and the set  of regular homotopy classes of codimension-1 immersions into $\RP^n$ \cite{B83}.

One of the aims of this paper is a contribution in this direction. Namely, we are interested in codimension-1 embeddings into a projective space. Our motivation is given by the following result, in the category of  complex manifolds  (see next section for the proof).

\begin{theorem}\label{mainteor1}
Let $M$ be a compact complex manifold of complex dimension $m \geq 1$.
There are no holomorphic embeddings of  $M \times \CP^n$ into $\CP^{m + n + 1}$ for $n>m$.
\end{theorem}

Here we ask if  the previous theorem has a topological counterpart,  by taking $M$ any
closed, connected, orientable, topological manifold and by considering
nonorientable, i.e. even dimensional, real projective spaces instead of complex projective spaces.
The following theorem  is a result in this direction.

\begin{theorem}\label{mainteor2}
Let $M$ be a closed, connected, orientable, topological manifold of dimension $m \geq 1$.
There are no topologically locally flat embeddings of $M \times \RP^n$ into $\RP^{m + n + 1}$ for all even $n > m$.
\end{theorem}

The paper is organized as follows. In the next section we  prove Theorem \ref{mainteor1} and Theorem \ref{mainteor2}. In the last one we show, by some counterexamples, that the above results do not hold by taking  immersions instead of embeddings, and that the assumptions on the dimension of $M$ can not be relaxed. Of course the Segre embedding $\CP^1 \times \CP^1 \to \CP^3$ provides an example of codimension-1 embedding between projective spaces with $m=n$, according to the notation of Theorem \ref{mainteor1}. However, in view of our results, one can ask if it exhausts all the embeddings of the product of projective spaces into a projective space. This will lead us to formulate a general conjecture on embeddings of Grassmannians.

\section{Proof of  the main results}
In order to prove Theorem \ref{mainteor1},
we recall some basic facts of algebraic geometry, referring the reader to Section 1 and 2 of Chapter 1 in  \cite{gh} for details and further results.

Let $X$ be a complex manifold of complex dimension $k$. Any co\-di\-men\-sion-1 complex submanifold $V\subset X$
can be seen as a smooth divisor on $X$. Recall that a divisor $D$ of $X$ is a locally finite formal linear combination $D=\sum a_jV_j$,
where the $V_j$ are irreducible holomorphic  subvarieties  of complex  dimension $n-1$ (equivalently, $V_j$ can be described locally as zeros of a single holomorphic function).
To each divisor $D$ on $X$ (and hence to  any codimension-1 complex submanifold $V\subset X$) we can associate a holomorphic line bundle $L=[V]$ on $X$ (see \cite[p.132]{gh}).
A holomorphic line bundle $L$ over a compact complex manifold $X$ is said to be {\em positive} if there exists an integral \K\ form $\omega$ on $X$ representing the first Chern class of $L$,
i.e. $c_1(L)=[\omega]$, where $[\omega]\in H_{dR}^2(X; \R)$ denotes the second de Rham cohomology class of $\omega$.
One has the following celebrated theorem.

\vskip 0.3cm

\noindent
{\bf Lefschetz Hyperplane Theorem.}
{\em Let $X$ be a  compact complex manifold of complex dimension $k$ and let $V$ be a  codimension-1 complex submanifold of $X$ such that the associated line bundle $[V]$ is positive.
Then, the linear map
$$H^q(X; \Q)\rightarrow H^q(V; \Q)$$
induced by the inclusion $V\hookrightarrow X$ is an isomorphism for $q\leq k-2$ and injective for $q=k-1$.}

We are now ready to prove our main theorems.

\begin{proof}[Proof of Theorem \ref{mainteor1}]
The proof is by contradiction. Suppose that there is a holomorphic embedding $\varphi \colon M \times \CP^n \to \CP^{m+n+1}$.
Since all the divisors of $\CP^{m+n+1}$ are multiple of the hyperplane divisor $H=\CP^{m+n}\subset \CP^{m+n+1}$,
it follows that the holomorphic line bundle $[V]$ associated to the  complex submanifold $V=\varphi (M\times \CP^n)\subset\CP^{m+n+1}$ is positive.
The assumptions $m\geq 1$ and $n>m$ together with the Lefschetz hyperplane theorem for $q=2$, yields the equality 
$$b_2(M\times \CP^n) = b_2(\CP^{m+n+1})$$
among the second Betti numbers.
Using  K\"{u}nneth's theorem and  the fact that the second Betti number of the complex projective space is 1, we get
$$1+b_2(M) = b_2(M\times \CP^n) = b_2(\CP^{m+n+1})=1,$$
yielding the desired contradiction since $b_2(M)\neq 0$ being $M$ an algebraic manifold.
\end{proof}
\begin{proof}[Proof of Theorem \ref{mainteor2}]
Suppose, by contradiction,  that there is a topologically locally flat embedding $g \colon M \times \RP^n \to \RP^{m+n+1}$. We put $N = g(M \times \RP^n)$ and $P_y = g(\{y\} \times \RP^n)$, where $y\in M$ is fixed. Since $g$ is an embedding, we have $N \cong M \times \RP^n$ and $P_y \cong \RP^n$. Let $p \colon S^{m+n+1} \to \RP^{m+n+1}$ be the universal covering map.

The preimage $\widetilde N = p^{-1}(N)$, being a closed codimension-1 locally flat submanifold of $S^{m+n+1}$, is the boundary of a codimension-0 submanifold $U \subset S^{m+n+1}$. Therefore, $\widetilde N$ is orientable, and since $M\times \RP^n$ is nonorientable because $n$ is even, it follows that $\widetilde N$ is connected.
We claim that the homomorphism
$$i_* \colon \pi_1(P_y) \to \pi_1(\RP^{m+n+1}) \cong \Z_2,$$
induced by the inclusion $i \colon P_y \to \RP^{m+n+1}$, is an isomorphism.

To prove the claim, we consider the monodromy $\omega \colon \pi_1(N) \to \Sigma_2$ of the covering map given by the restriction $p_| \colon \widetilde N \to N$, where $\Sigma_2$ is the symmetric group of two elements. Consider also the monodromy $\tilde \omega \colon \pi_1(\RP^{m+n+1}) \to \Sigma_2$ of the covering map $p$. It is clear that $\omega = \tilde\omega \circ j_*$, where $j_* \colon \pi_1(N) \to \pi_1(\RP^{m+n+1})$ is the homomorphism induced by the inclusion $j \colon N \to \RP^{m+n+1}$.

Let $a \in \pi_1(M)$ be any element, and let $b \in \pi_1(\RP^n) \cong \Z_2$ be the generator.
A loop in $\RP^n$ that represents $b$ is orientation-reversing (which means that the parallel transport along this loop reverses the orientation of the tangent space at the base point, or, equivalently, that the tubular neighborhood of this loop in $\RP^n$ is a nonorientable manifold). On the other hand, since $M$ is orientable, every loop in $M$ is orientation-preserving (that is, its tubular neighborhood is orientable). Hence, a loop $\lambda$ in $N$ that represents $(a, b) \in \pi_1(M) \times \pi_1(\RP^n) \cong \pi_1(N)$ is orientation-reversing.

Since $\widetilde N$ is orientable, every lifting $\tilde \lambda$ of $\lambda$ to $\widetilde N$ cannot be a loop (otherwise the tubular neighborhood of $\tilde\lambda$ in $\widetilde N$ would map homeomorphically to the tubular neighborhood of $\lambda$ by $p_{|\widetilde N}$, which is impossible because of the nonorientability of the latter). This implies that $\omega([\lambda]) = \omega(a, b)$ is not the identity element of $\Sigma_2$. In particular, for $a = 1$, $\omega(1, b) = \tilde\omega(j_*(1,b)) = \tilde\omega(i_*(b)) \ne 1$, and so $i_*(b) \ne 1$. Since $\pi_1(P_y) \cong \Z_2$, this proves the claim.

Now, we proceed with the proof of the theorem. The homology class of $P_y$ in $$H_{n}(\RP^{m+n+1}; \Z_2) \cong \Z_2$$ is equal to zero because $P_y$ is disjoint from an isotopic copy of it, say $P_{y'}$ for $y' \neq y$, while the generator of $H_{n}(\RP^{m+n+1}; \Z_2)$, that is the homology class of the standard $\RP^{n} \subset \RP^{m+n+1}$, has non-zero self-intersection, given that $n > m$.

Let $\eta \in H^1(\RP^{m+n+1}; \Z_2) \cong \Z_2$ be the generator. The fact that $i_*$ is an isomorphism between fundamental groups implies that $i^*(\eta) \neq 0$ in $H^1(P_y; \Z_2)$. Thus, $i^*(\eta)^{n} \neq 0$ in $H^{n}(P_y; \Z_2)$. On the other hand, $i^*(\eta)^{n} = i^*(\eta^{n}) = 0$ because $\langle \eta^n, [P_y]\rangle = 0$, being $[P_y] = 0$ in $H_{n}(\RP^{m+n+1}; \Z_2)$. Having obtained a contradiction, we conclude the proof.
\end{proof}

\begin{corollary}
There are no smooth embeddings of $M \times \RP^{n}$ into $\RP^{m+n+1}$ for every smooth, closed, orientable manifold $M$ of dimension $m \geq 1$ and for every even $n > m$.
\end{corollary}

\section{Final remarks}

\begin{enumerate}[1.]

\item The assumption $n>m$ in Theorem \ref{mainteor1} cannot be  relaxed.
For example, $\CP^1\times\CP^1$ admits a  holomorphic embedding into  $\CP^3$
by means of the Segre embedding, namely the map
\begin{align*}
\CP^1 \times \CP^1 &\to \CP^3\\ ([z_0, z_1], [w_0, w_1]) &\mapsto [z_0w_0, z_0w_1, z_1w_0, z_1w_1].
\end{align*}
Nevertheless, we do not know if the hypothesis $n>m$ in Theorem \ref{mainteor2} (still assuming $n$ even) can be dropped (cf. the proof given in the previous section).
\medskip

\item Let us consider a smooth embedding of $S^1\times \RP^3$ into $\RP^5$ constructed as follows. First, take an embedding $\psi\colon\RP^3\rightarrow\R^5$, whose existence is guaranteed by Wall's theorem \cite{W65}; then consider a compact tubular neighborhood
$T$ of $\psi (\RP^3)$ in $\R^5$. Since $\RP^3$ is orientable and the embedding is of codimension 2, it is well known \cite[Chapter 11]{MS74} that the normal bundle of $\RP^3$ in $\R^5$ is trivial, hence $T\cong D^2\times \RP^3$, where $D^2$  denotes the $2$-dimensional disk. By taking the boundary of $T$, we then get an embedding of $S^1\times \RP^3$ into $\R^5$. Finally, by composing this embedding  with the inclusion of $\R^5$ into $\RP^5$  as an affine chart one gets a smooth embedding $S^1\times \RP^3\rightarrow \RP^5$. Similarly, starting from a suitable embedding of $\RP^3$ into $\R^6$ (obtained by composing the previous one with the standard inclusion $\R^5 \subset \R^6$), one can show that there exists a smooth embedding $S^2\times \RP^3\rightarrow \RP^6$.
These constructions show that the evenness of $n$ in Theorem \ref{mainteor2} cannot be avoided.
\medskip

\item The assumption in Theorem \ref{mainteor2} cannot be weakened by considering  immersions instead of embeddings. Indeed,
$S^1\times \RP^2$ (actually $S^1\times S$ for any closed surface $S$) admits a smooth immersion into $\RP^4$ constructed as follows. Take an embedding $\psi \colon \RP^2\hookrightarrow \R^4$
and a $2$-dimensional plane $H$ disjoint from $\psi (\RP^2)$;  then, the rotation of $\psi (\RP^2)$ around $H$ generates an immersed copy of $S^1\times \RP^2$. By composing this immersion with the inclusion of $\R^4$ into $\RP^4$  as an affine chart, one gets the desired immersion.
\medskip

\end{enumerate}

Notice that the  proof of Theorem \ref{mainteor1} easily extends  by taking complex Grassmannians instead of projective spaces. Thus, we believe that Theorem \ref{mainteor2} can be extended to nonorientable Grassmannians instead of projective spaces.

Actually, in view of this last observation and the first remark in this section, it makes sense to formulate the following more general conjecture.

\begin{conjecture*}
Let $G_{k_i}(\R^{n_i})$ denote the Grassmannian of $k_i$-planes in $\R^{n_i}$, for $k_i \geq 1$ and $i\in\left\{1,2,3\right\}$. Then, the only codimension-$1$ smooth embedding
\begin{equation*}
G_{k_1}(\R^{n_1})\times  G_{k_2}(\R^{n_2}) \rightarrow G_{k_3}(\R^{n_3})
\end{equation*}
is the Segre embedding
\begin{equation*}
\RP^1 \times \RP^1 \rightarrow \RP^3.
\end{equation*}
\end{conjecture*}

\end{document}